\documentclass[11pt]{article}

\input{style}

\begin{document}

\title{Existence and uniqueness of Arrow-Debreu equilibria with
  consumptions in $\mathbf{L}^0_+$}

\author{Dmitry Kramkov\thanks{This research was supported in part by
    the Carnegie Mellon-Portugal Program and by the Oxford-Man
    Institute for Quantitative Finance at the University of Oxford.}\\
  Carnegie Mellon University and University of Oxford,\\
  Department of Mathematical Sciences,\\ 5000 Forbes Avenue,
  Pittsburgh, PA, 15213-3890, USA}

\date{\today}
\maketitle
\begin{abstract}
  We consider an economy where agents' consumption sets are given by
  the cone $\mathbf{L}^0_+$ of non-negative measurable functions and
  whose preferences are defined by additive utilities satisfying the
  Inada conditions. We extend to this setting the results in
  \citet{Dana:93} on the existence and uniqueness of Arrow-Debreu
  equilibria. In the case of existence, our conditions are necessary
  and sufficient.
\end{abstract}

\begin{description}
\item[MSC:] 91B50, 91B51.
\item[Keywords:] Arrow-Debreu equilibrium, Pareto allocation.
\end{description}

\section{Main results}
\label{sec:results}

This paper states necessary and sufficient conditions for the
existence of Arrow-Debreu equilibria in the economy where agents'
consumption sets are given by the cone $\mathbf{L}^0_+$ of
non-negative measurable functions and whose preferences are defined by
additive utilities satisfying the Inada conditions. For completeness,
a version of a well-known uniqueness criteria is also provided.  The
results generalize those in \citet{Dana:93}; see Remarks~\ref{rem:2}
and~\ref{rem:3} for details.

Consider an economy with $M\in \braces{1,2,\dots}$ agents and a state
space $(\mathbf{S}, \mathcal{S}, \mu)$ which is a measure space with a
$\sigma$-finite measure $\mu$. By $\mathbf{L}^0$ we denote the space
of (equivalence classes of) measurable functions with values in
$[-\infty,\infty)$; $\mathbf{L}^0_+$ stands for the cone of
non-negative measurable functions.  For $\alpha\in \mathbf{L}^0$ we
use the convention:
\begin{displaymath}
  \mathbb{E}[\alpha] \set \int_{\mathbf{S}} \alpha(s) \mu(ds) \set
  -\infty \quad 
  \text{if} \quad 
  \mathbb{E}[\min(\alpha,0)]=-\infty. 
\end{displaymath}
The {consumption set} of $m$th agent equals $\mathbf{L}^0_+$ and his
preferences are defined by additive utility:
\begin{displaymath}
  u_m(\alpha^m) \set \mathbb{E}[U_m(\alpha^m)], \quad \alpha^m\in
  \mathbf{L}^0_+. 
\end{displaymath}
The utility measurable field $\map{U_m}{[0,\infty)}{\mathbf{L}^0}$ has
the following properties.

\begin{Assumption}
  \label{as:1}
  For every $s\in \mathbf{S}$ the function $U_m(\cdot)(s)$ on
  $(0,\infty)$ is strictly increasing, strictly concave, continuously
  differentiable, and satisfies the Inada conditions:
  \begin{displaymath}
    \lim_{c\to \infty}U'_m(c)(s) = 0, \quad  \lim_{c\downarrow
      0}U'_m(c)(s) = \infty, 
  \end{displaymath}
  At $c=0$, by continuity, $U_m(0)(s) = \lim_{c\downarrow 0}
  U_m(c)(s)$; this limit may be $-\infty$.
\end{Assumption}

\begin{Remark}
  \label{rem:1}
  This framework readily accommodates consumption preferences which
  are common in Mathematical Finance. For instance, it includes a
  continuous-time model defined on a filtered probability space
  $(\Omega, \mathcal{F}, (\mathcal{F}_t)_{t\geq 0}, \mathbb{P})$ where
  the agents consume according to an optional non-negative process
  $\alpha=(\alpha_t)$ of consumption rates; so that $C_t = \int_0^t
  \alpha_v dv$ is the cumulative consumption up to time $t$.  In this
  case, the state space $\mathbf{S} = \Omega \times [0,\infty)$,
  $\mathcal{S}$ is the optional $\sigma$-algebra, and $\mu(ds) = \mu(d
  \omega, dt) = \mathbb{P}[d\omega] \times dt$.
\end{Remark}

By $\Lambda\in \mathbf{L}^0_+$ we denote the total endowment in the
economy.  A family $(\alpha^m)_{m=1,\dots,M}$ of elements of
$\mathbf{L}^0_+$ such that $\sum_{m=1}^M \alpha^m =\Lambda$ is called
\emph{an allocation of $\Lambda$}. By $(\alpha^m_0)\subset
\mathbf{L}^0_+$ we denote the \emph{initial allocation} of $\Lambda$
between the agents.

\begin{Definition}
  \label{def:1}
  A pair $(\zeta, (\alpha^m_1)_{m=1,\dots,M})$, consisting of a
  measurable function $\zeta>0$ (a \emph{state price density}) and an
  allocation $(\alpha^m_1)$ of $\Lambda$ (an \emph{equilibrium
    allocation}) is an \emph{ Arrow-Debreu equilibrium} if
  \begin{displaymath}
    \mathbb{E}[\zeta \Lambda] <\infty,
  \end{displaymath}
  and, for $m=1,\dots,M$,
  \begin{gather}
    \nonumber
    \mathbb{E}[\abs{U_m(\alpha^m_1)}] < \infty, \\
    \label{eq:1}
    \alpha^m_1 = \argmax \descr{\mathbb{E}[U_m(\alpha)]}{\alpha\in
      \mathbf{L}^0_+, \quad \mathbb{E}[\zeta(\alpha-\alpha^m_0)]=0}.
  \end{gather}
\end{Definition}

This is the main result of the paper.

\begin{Theorem}
  \label{th:1}
  Suppose that Assumption~\ref{as:1} holds, the total endowment
  $\Lambda>0$, and the initial allocation $\alpha^m_0\not=0$,
  $m=1,\dots,M$.  Then an Arrow-Debreu equilibrium exists if and only
  if there is an allocation $(\alpha^m)$ of $\Lambda$ such that
  \begin{equation}
    \label{eq:2}
    \mathbb{E}[\abs{U_m(\alpha^m)} + U'_m(\alpha^m)\Lambda] <
    \infty, \quad m=1,\dots,M. 
  \end{equation}
\end{Theorem}

\begin{Corollary}
  \label{cor:1}
  Suppose that Assumption~\ref{as:1} holds, the total endowment
  $\Lambda>0$, the initial allocation $ \alpha^m_0\not=0$,
  $m=1,\dots,M$, and
  \begin{displaymath}
    \mathbb{E}[\abs{U_m(s\Lambda)}] < \infty, \quad s\in (0,1], \; 
    m=1,\dots,M.  
  \end{displaymath}
  Then an Arrow-Debreu equilibrium exists.
\end{Corollary}
\begin{proof}
  The concavity of $U_m$ implies that
  \begin{displaymath}
    cU'_m(c) \leq 2(U_m(c) - U_m(c/2)), \quad c>0,    
  \end{displaymath}
  and the result follows from Theorem~\ref{th:1} where one can take
  $\alpha^m = \frac1M \Lambda$.
\end{proof}

The Arrow-Debreu equilibrium is, in general, not unique; see,
e.g. Example~15.B.2 in \citet{MasColellWhinstonGreen:95} which can be
easily adapted to our setting.  We state a version of a well-known
uniqueness criteria.

\begin{Theorem}
  \label{th:2}
  Suppose that the conditions of Corollary~\ref{cor:1} hold and there
  is an index $m_0$ such that for $m\not=m_0$
  \begin{equation}
    \label{eq:3}
    \text{the map $c\mapsto cU'_m(c)$ of $(0,\infty)$ to $\mathbf{L}^0_{+}$ is
      non-decreasing.}
  \end{equation}
  Then an Arrow-Debreu equilibrium exists and is unique.
\end{Theorem}

The proofs of Theorems~\ref{th:1} and~\ref{th:2} are given in
Section~\ref{sec:existence} and rely on the finite-dimensional
characterization of Arrow-Debreu equilibria in Theorem~\ref{th:4}.

\begin{Remark}
  \label{rem:2}
  Being necessary and sufficient, the condition~\eqref{eq:2} improves
  the sufficient criteria in~\citet{Dana:93}. This paper requires the
  existence of conjugate exponents $p,q\in[1,\infty]$, $\frac1p +
  \frac1q=1$, and measurable functions $\zeta \in \mathbf{L}^q_+$ and
  $\eta\in \mathbf{L}^1_+$ such that $\Lambda\in \mathbf{L}^p_+$ and
  the measurable fields $(U_m)$ are twice continuously differentiable
  and satisfy
  \begin{displaymath}
    U_m(c) \leq \zeta c + \eta, \quad c>0, \; m=1,\dots,M, 
  \end{displaymath}
  and
  \begin{displaymath}
    \sup_{w\in \Sigma^M}U_c(w,\Lambda) \Lambda \in
    \mathbf{L}^q_+. 
  \end{displaymath}
  Here $U(w,c)$ is the aggregate utility measurable field,
  see~\eqref{eq:5} below, and $\Sigma^M$ is the simplex in
  $\mathbf{R}^M$.
\end{Remark}

\begin{Remark}
  \label{rem:3}
  If the utility measurable fields $(U_m)$ are twice differentiable,
  then the key condition~\eqref{eq:3} is equivalent to the boundedness
  by $1$ of their relative risk-aversions $-cU_m''(c)/U_m'(c)$ and, in
  this form, is well-known; see, e.g., Example~17.F.2 and
  Proposition~17.F.3 in \citet{MasColellWhinstonGreen:95},
  Theorem~4.6.1 in \citet{KaratShr:98}, and \citet{Dana:93}.  Note
  that, contrary to the above references, we allow this condition to
  fail for one economic agent.
\end{Remark}

\section{Pareto optimal allocations}
\label{sec:pareto-optim-alloc}

In this section we state the properties of Pareto optimal allocations
needed for the proofs of Theorems~\ref{th:1} and~\ref{th:2}.

\begin{Definition}
  \label{def:2}
  An allocation $(\alpha^m)$ of $\Lambda$ is \emph{Pareto optimal} if
  \begin{equation}
    \label{eq:4}
    \mathbb{E}[\abs{U_m(\alpha^m)}] < \infty, \quad m=1,\dots,M, 
  \end{equation}
  and there is no allocation $(\beta^m)$ of $\Lambda$ which dominates
  $(\alpha^m)$ in the sense that
  \begin{displaymath}
    \mathbb{E}[U_m(\beta^m)] \geq \mathbb{E}[U_m(\alpha^m)],\quad
    m=1,\dots,M, 
  \end{displaymath}
  and
  \begin{displaymath}
    \mathbb{E}[U_{l}(\beta^{l})] >
    \mathbb{E}[U_{l}(\alpha^{l})] \quad\text{for some}\quad {l}\in
    \braces{1,\dots,M}. 
  \end{displaymath}
\end{Definition}

See Remark~\ref{rem:4} for a justification of the integrability
condition~\eqref{eq:4}.

As usual, in the study of complete equilibria, an important role is
played by the \emph{aggregate utility measurable field}
$\map{U=U(w,c)}{\mathbf{R}^M_{+}\times [0,\infty)}{\mathbf{L}^0}$:
\begin{equation}
  \label{eq:5}
  U(w,c) \set \sup\descr{\sum_{m=1}^M w^m U_m(c^m)}{c^m\geq 0, \quad
    c^1+\cdots+c^M=c}.
\end{equation}
Here $\mathbf{R}^M_+$ is the non-negative orthant in $\mathbf{R}^M$
without the origin:
\begin{displaymath}
  \mathbf{R}^M_+ \set \descr{w\in [0,\infty)^M}{\sum_{m=1}^M w^m >0}. 
\end{displaymath}
Due to the $1$-homogeneity property of $U=U(w,c)$ with respect to the
weight $w$: $U(yw,c) = yU(w,c)$, $y>0$, it is often convenient to
restrict its $w$-domain to the simplex
\begin{displaymath}
  \Sigma^M \set \descr{w\in [0,\infty)^M}{\sum_{m=1}^M w^m =1}.
\end{displaymath}

Elementary arguments show that for every $w\in \mathbf{R}^M_+$ the
measurable field $U(w,\cdot)$ on $[0,\infty)$ satisfies
Assumption~\ref{as:1} and that the upper bound in~\eqref{eq:5} is
attained at the measurable functions $\pi^m(w,c)$, $m=1,\dots,M$, such
that
\begin{align}
  \label{eq:6}
  w^m U_m'(\pi^m(w,c)) &= U_c(w,c) \set \frac{\partial}{\partial
    c}U(w,c) \quad \text{if} \quad w^m>0,\; c>0, \\
  \label{eq:7}
  \pi^m(w,c) &= 0 \quad \text{if} \quad w^m=0 \quad\text{or}\quad c=0.
\end{align}

These identities readily imply that the measurable fields
$\map{\pi^m}{\mathbf{R}^M_{+}\times [0,\infty)}{\mathbf{L}^0_+}$,
$m=1,\dots,M$, are continuous and have the following properties:

\begin{enumerate}[label=(A\arabic{*}), ref=(A\arabic{*})]
\item \label{item:1} They are $0$-homogeneous with respect to $w$:
  $\pi^m(yw,c) = \pi^m(w,c)$, $y>0$.
\item \label{item:2} For $c>0$ the measurable field $\pi^m(\cdot,c) =
  \pi^m(w^1,\dots,w^M,c)$ is strictly increasing with respect to $w^m$
  and strictly decreasing with respect to $w^l$, $l\not=m$.
\item \label{item:3} If $c>0$ and $w_1$ and $w_2$ are distinct
  vectors in $\Sigma^M$, then for every $s \in \mathbf{S}$ the vectors
  $(\pi^m(w_1,c)(s))$ and $(\pi^m(w_2,c)(s))$ in $c\Sigma^M$ are
  distinct.
\end{enumerate}

We state a version of the well-known parametrization of Pareto optimal
allocations by the elements of $\Sigma^M$.

\begin{Theorem}
  \label{th:3}
  Suppose that Assumption~\ref{as:1} holds. Let $\Lambda>0$ and
  $(\alpha^m)$ be an allocation of $\Lambda$ satisfying~\eqref{eq:4}.
  Then $(\alpha^m)$ is Pareto optimal if and only if there is $w\in
  \Sigma^M$ such that
  \begin{equation}
    \label{eq:8}
    \alpha^m = \pi^m(w,\Lambda), \quad m=1,\dots,M. 
  \end{equation}
\end{Theorem}
\begin{proof}
  Assume first that $(\alpha^m)$ is Pareto optimal. By the concavity
  of the measurable fields $(U_m)$, the set
  \begin{displaymath}
    C\set \descr{z\in \mathbf{R}^M}{z^m \leq \mathbb{E}[U_m(\beta^m)]
      \; \text{for some allocation $(\beta^m)$ of $\Lambda$}} 
  \end{displaymath}
  is convex and, in view of~\eqref{eq:4}, it has a non-empty
  interior. From the Pareto optimality of $(\alpha^m)$ we deduce that
  the point
  \begin{displaymath}
    \widehat z^m \set \mathbb{E}[U_m(\alpha^m)], \quad m=1\dots, M,
  \end{displaymath}
  belongs to the boundary of $C$. Hence, there is a non-zero $w\in
  \mathbf{R}^M$ such that
  \begin{equation}
    \label{eq:9}
    \sum_{m=1}^M w^m\widehat z^m  \geq  \sum_{m=1}^M w^m z^m, \quad z\in C. 
  \end{equation}
  Since $C-[0,\infty)^M = C$, we obtain that $w^m\geq 0$. Then we can
  normalize $w$ to be in $\Sigma^M$. Observe now that~\eqref{eq:9} can
  be written as
  \begin{align*}
    \mathbb{E}[\sum_{m=1}^M w^m U_m(\alpha^m)] &= \sup
    \descr{\mathbb{E}[\sum_{m=1}^M w^m
      U_m(\beta^m)]}{\text{$(\beta^m)$ is an allocation of
        $\Lambda$}} \\
    &= \mathbb{E}[U(w,\Lambda)],
  \end{align*}
  which readily implies~\eqref{eq:8}.

  Conversely, if $(\alpha^m)$ is given by~\eqref{eq:8}, then for any
  allocation $(\beta^m)$ of $\Lambda$
  \begin{align*}
    \sum_{m=1}^M w^m U_m(\beta^m) \leq U(w,\Lambda) = \sum_{m=1}^M w^m
    U_m(\alpha^m),
  \end{align*}
  which yields the Pareto optimality of $(\alpha^m)$ after we account
  for the integrability condition~\eqref{eq:4} and recall that
  $\alpha^m\set \pi^m(w,\Lambda)=0$ if $w^m=0$.
\end{proof}

\begin{Remark}
  \label{rem:4}
  Without the integrability condition~\eqref{eq:4} in
  Definition~\ref{def:2} the assertion of Theorem~\ref{th:3} does not
  hold. As a counter-example, take $M=2$ and select the total
  endowment and the utility functions so that $\Lambda>0$, $U_1=U_2$,
  and
  \begin{displaymath}
    \mathbb{E}[\abs{U_i(\Lambda)}] < \infty \quad\text{and}\quad
    \mathbb{E}[U_i(\Lambda/2)] = - \infty.
  \end{displaymath}
  Then the allocation $(\pi_1(1/2,\Lambda),\pi_2(1/2,\Lambda)) =
  (\Lambda/2,\Lambda/2)$ is dominated by either $(0,\Lambda)$ or
  $(\Lambda,0)$.
\end{Remark}

After these preparations, we are ready to state the main result of
this section.

\begin{Theorem}
  \label{th:4}
  Suppose that Assumption~\ref{as:1} holds, the total endowment
  $\Lambda>0$, and the initial allocation $\alpha^m_0\not=0$,
  $m=1,\dots,M$. Then $(\zeta,(\alpha^m_1))$ is an Arrow-Debreu
  equilibrium if and only if there is $w\in \interior{\Sigma^M}$, the
  interior of $\Sigma^M$, and a constant $z>0$ such that
  \begin{align}
    \label{eq:10}
    \zeta &= z U_c(w,\Lambda),
    \\
    \label{eq:11}
    \alpha^m_1 &= \pi^m(w,\Lambda), \quad m=1,\dots,M,
  \end{align}
  and
  \begin{gather}
    \label{eq:12}
    \mathbb{E}[\sum_{m=1}^M \abs{U_m(\pi^m(w,\Lambda))} +
    U_c(w,\Lambda)\Lambda] < \infty, \\
    \label{eq:13}
    \mathbb{E}[U_c(w,\Lambda)(\pi^m(w,\Lambda) - \alpha^m_0)]=0, \quad
    m=1,\dots,M.
  \end{gather}
  In this case, the equilibrium allocation $(\alpha^m_1)$ is Pareto
  optimal.
\end{Theorem}

\begin{Remark}
  \label{rem:5}
  As Lemma~\ref{lem:2} below shows, the integrability
  condition~\eqref{eq:12} holds for some $w\in \interior{\Sigma^M}$ if
  and only if it holds for every $w\in \interior{\Sigma^M}$ and is
  equivalent to the existence of an allocation $(\alpha^m)$ of
  $\Lambda$ satisfying~\eqref{eq:2}.
\end{Remark}

For the proof of the theorem we need a lemma.

\begin{Lemma}
  \label{lem:1}
  Let $\map{U}{[0,\infty)}{\mathbf{L}^0}$ be a measurable field
  satisfying Assumption~\ref{as:1} and $\alpha>0$ and $\zeta>0$ be
  measurable functions such that $\mathbb{E}[\zeta \alpha]<\infty$,
  and
  \begin{displaymath}
    -\infty < \mathbb{E}[U(\alpha)] =
    \sup\descr{\mathbb{E}[U(\beta)]}{\beta\in \mathbf{L}^0_+, \;
      \mathbb{E}[\zeta(\beta - \alpha)] = 0} < \infty. 
  \end{displaymath}
  Then
  \begin{equation}
    \label{eq:14}
    \zeta  =  \frac{\mathbb{E}[\zeta
      \alpha]}{\mathbb{E}[U'(\alpha)\alpha]} U'(\alpha).  
  \end{equation}
\end{Lemma}
\begin{proof}
  Let $\beta$ be a measurable function such
  that 
  \begin{displaymath}
    0<\beta<\alpha \quad \text{and} \quad
    \mathbb{E}[U(\alpha-\beta)]>-\infty.
  \end{displaymath}
  For instance, we can choose $\beta$ so that
  \begin{displaymath}
    U(\alpha) - U(\alpha-\beta) = \min(\eta, U(2\alpha)-U(\alpha)), 
  \end{displaymath}
  where the measurable function $\eta>0$ and
  $\mathbb{E}[\eta]<\infty$.  Observe that
  \begin{equation}
    \label{eq:15}
    \mathbb{E}[\zeta \beta]<\infty \quad \text{and} \quad
    \mathbb{E}[U'(\alpha)\beta]<\infty, 
  \end{equation}
  where the second bound holds because
  \begin{displaymath}
    U'(\alpha)\beta \leq U(\alpha) - U(\alpha-\beta). 
  \end{displaymath}

  Take a measurable function $\gamma$ such that
  \begin{displaymath}
    \abs{\gamma}\leq \beta \quad \text{and} \quad
    \mathbb{E}[\zeta\gamma] = 0. 
  \end{displaymath}
  From the properties of $U$ in Assumption~\ref{as:1} we deduce that
  for $\abs{y}<1$:
  \begin{align*}
    U'(\alpha + y\gamma/2) \abs{\gamma}/2 &\leq U'(\alpha - \beta/2)
    {\beta}/2 \\
    &\leq U(\alpha-\beta/2) - U(\alpha-\beta) \\
    &\leq U(\alpha) - U(\alpha-\beta).
  \end{align*}
  Accounting for the integrability of $U(\alpha)$ and
  $U(\alpha-\beta)$ we obtain that the function
  \begin{displaymath}
    f(y) \set \mathbb{E}[U(\alpha+y\gamma/2)], \quad y\in
    (-1,1), 
  \end{displaymath}
  is continuously differentiable and
  \begin{displaymath}
    f'(y) = \mathbb{E}[U'(\alpha+y\gamma/2)\gamma/2], \quad y\in
    (-1,1). 
  \end{displaymath}
  As $\alpha$ is optimal, $f$ attains its maximum at $y=0$. Therefore,
  \begin{displaymath}
    f'(0) = \mathbb{E}[U'(\alpha)\gamma] = 0. 
  \end{displaymath}

  Thus, we have found a measurable function $\beta>0$ such
  that~\eqref{eq:15} holds and
  \begin{equation}
    \label{eq:16}
    \forall \gamma\in \mathbf{L}^0:\quad \abs{\gamma}\leq \beta \quad
    \text{and} \quad 
    \mathbb{E}[\zeta\gamma] = 0 \quad \Rightarrow \quad
    \mathbb{E}[U'(\alpha)\gamma] = 0. 
  \end{equation}
  This readily implies~\eqref{eq:14}. Indeed, if we define the
  probability measures $\mathbb{P}$ and $\mathbb{Q}$ as
  \begin{displaymath}
    \frac{d\mathbb{P}}{d\mu} \set \frac{\zeta \beta}{\mathbb{E}[\zeta
      \beta]} \quad \text{and} \quad \frac{d\mathbb{Q}}{d\mu} \set
    \frac{U'(\alpha) \beta}{\mathbb{E}[U'(\alpha) \beta]},
  \end{displaymath}
  then for a set $A\in \mathcal{S}$ the implication~\eqref{eq:16} with
  $\gamma = \beta(\setind{A} - \mathbb{P}[A])$ yields that $\mathbb{P}[A]
  = \mathbb{Q}[A]$. Hence, $\mathbb{P} = \mathbb{Q}$ and the result
  follows.
\end{proof}

\begin{proof}[Proof of Theorem~\ref{th:4}.] Assume first that
  $(\zeta,(\alpha^m_1))$ is an Arrow-Debreu equilibrium.  We claim
  that $(\alpha^m_1)$ is a Pareto optimal allocation of
  $\Lambda$. Indeed, if there is an allocation $(\beta^m)$ of
  $\Lambda$ which dominates $(\alpha^m_1)$ in the sense of
  Definition~\ref{def:2} then for some $l\in \braces{1,\dots,M}$
  either
  \begin{displaymath}
    \mathbb{E}[U_l(\beta^l)] \geq \mathbb{E}[U_l(\alpha^l_1)]
    \quad\text{and}\quad \mathbb{E}[\zeta\beta^l] < \mathbb{E}[\zeta
    \alpha^l_1] 
  \end{displaymath}
  or
  \begin{displaymath}
    \mathbb{E}[U_l(\beta^l)] > \mathbb{E}[U_l(\alpha^l_1)]
    \quad\text{and}\quad \mathbb{E}[\zeta\beta^l] \leq
    \mathbb{E}[\zeta \alpha^l_1], 
  \end{displaymath}
  contradicting the optimality of $\alpha^l_1$ in~\eqref{eq:1}.

  From Theorem~\ref{th:3} we then obtain the
  representation~\eqref{eq:11} for $(\alpha^m_1)$ with $w\in
  \Sigma^M$. As $\alpha^m_0\not=0$, we have $\alpha^m_1\not=0$. Hence,
  $w\in \interior{\Sigma^M}$ and $\alpha^m_1 = \pi^m(w,\Lambda)>0$.
  Lemma~\ref{lem:1} and the identities~\eqref{eq:6} now
  imply~\eqref{eq:10}. Finally, the properties~\eqref{eq:12}
  and~\eqref{eq:13} are parts of Definition~\ref{def:1}.

  Conversely, let $w\in \interior{\Sigma^M}$ be such that the
  conditions \eqref{eq:12} and \eqref{eq:13} hold and define $(\zeta,
  (\alpha^m_1))$ by \eqref{eq:10} and
  \eqref{eq:11}. Except~\eqref{eq:1} all conditions of
  Definition~\ref{def:1} hold trivially. The property~\eqref{eq:1}
  follows from the inequalities
  \begin{align*}
    U_m(\beta^m) - \frac1{w^m}U_c(w,\Lambda) \beta^m &\leq
    \sup_{c\geq 0} \braces{U_m(c) - \frac1{w^m}U_c(w,\Lambda)c} \\
    &= U_m(\pi^m(w,\Lambda)) - \frac1{w^m} U_c(w,\Lambda)
    \pi^m(w,\Lambda),
  \end{align*}
  where $\beta^m\in \mathbf{L}^0_+$ and, at the last step, we used the
  definition of $\pi^m$ in~\eqref{eq:6}.
\end{proof}

\section{Proofs of Theorems~\ref{th:1} and \ref{th:2}}
\label{sec:existence}

We begin with some lemmas.

\begin{Lemma}
  \label{lem:2}
  Suppose that Assumption~\ref{as:1} holds, the total endowment
  $\Lambda>0$, and there is an allocation $(\alpha^m)$ of $\Lambda$
  satisfying~\eqref{eq:2}. Then
  \begin{gather}
    \label{eq:17}
    \mathbb{E}[\abs{U_m(\pi^m(w,\Lambda))}] < \infty, \quad
    w\in \interior{\Sigma^M}, \; m=1,\dots,M, \\
    \label{eq:18}
    \mathbb{E}[\sup_{w\in \Sigma^M}U_c(w,\Lambda)\Lambda] < \infty.
  \end{gather}
\end{Lemma}
\begin{proof}
  First, take $w\in \interior{\Sigma^M}$. From the monotonicity and
  the concavity of $U_m$ we deduce that
  \begin{displaymath}
    U_m(\pi^m(w,\Lambda)) \leq U_m(\Lambda) \leq U_m(\alpha^m) +
    U'_m(\alpha^m)\Lambda 
  \end{displaymath}
  and from the definition of $\pi^m$ that
  \begin{displaymath}
    \sum_{m=1}^M w^m U_m(\alpha^m) \leq U(w,\Lambda) = \sum_{m=1}^M
    w^m U_m(\pi^m(w,\Lambda)). 
  \end{displaymath}
  These inequalities readily imply~\eqref{eq:17}.

  Assume now that $w\in \Sigma^M$. Since
  \begin{displaymath}
    0< \Lambda = \sum_{m=1}^M \pi^m(w,\Lambda) = \sum_{m=1}^M \alpha^m, 
  \end{displaymath}
  for every $s\in \mathbf{S}$ there is an index $m(s)$ such that
  \begin{displaymath}
    0< \pi^{m(s)}(w,\Lambda)(s) \quad \text{and} \quad
    \alpha^{m(s)}(s) \leq \pi^{m(s)}(w,\Lambda)(s).  
  \end{displaymath}
  Accounting for~\eqref{eq:6} and~\eqref{eq:7} we deduce that
  \begin{displaymath}
    U_c(w,\Lambda) \leq \max\descr{w^m U'_m(\alpha^m)}{w^m>0}
    \leq \sum_{m=1}^M U'_m(\alpha^m), 
  \end{displaymath}
  which implies~\eqref{eq:18}.
\end{proof}

The following corollary of the Brouwer's fixed point theorem plays a
key role. It is well-known, see, e.g., Theorem~6.3.6 in
\citet{DanJeanb:03}.

\begin{Lemma}
  \label{lem:3}
  Let $\map{f_m}{\Sigma^M}{\mathbf{R}}$, $m=1,\dots,M$, be continuous
  functions such that
  \begin{gather}
    \label{eq:19}
    f_m(w) <0 \quad\text{if}\quad w^m=0, \\
    \label{eq:20}
    \sum_{m=1}^M f_m(w) = 0, \quad w\in \Sigma^M.
  \end{gather}
  Then there is $\widehat w\in \interior{\Sigma^M}$ such that
  \begin{equation}
    \label{eq:21}
    f_m(\widehat w) = 0, \quad m=1,\dots,M, 
  \end{equation}  
\end{Lemma}
\begin{proof}
  Since
  \begin{displaymath}
    g_m(w) \set \frac{w^m + \max(0,-f_m(w))}{1 + \sum_{n=1}^M
      \max(0,-f_n(w))}, \quad m=1,\dots,M,  
  \end{displaymath}
  are continuous maps of $\Sigma^M$ into $\Sigma^M$, the Brouwer's
  fixed point theorem implies the existence of $\widehat w\in
  \Sigma^M$ such that
  \begin{displaymath}
    \widehat w^m = g_m(\widehat w) =  \frac{\widehat{w}^m +
      \max(0,-f_m(\widehat w))}{1 + \sum_{n=1}^M
      \max(0,-f_n(\widehat w))}, \quad m=1,\dots,M,
  \end{displaymath}

  In view of~\eqref{eq:19}, we obtain that $\widehat w \in
  \interior{\Sigma^M}$. From \eqref{eq:20} we deduce the existence of
  an index $l$ such that $f_l(\widehat w)\geq 0$. Then
  \begin{displaymath}
    0 < \widehat w^l = \frac{\widehat{w}^l}{1 + \sum_{m=1}^M
      \max(0,-f_m(\widehat w))}, 
  \end{displaymath}
  implying that $f_m(\widehat w) \geq 0$, $m=1,\dots,M$. Another use
  of~\eqref{eq:20} yields the result.
\end{proof}

\begin{proof}[Proof of Theorem~\ref{th:1}.]  The ``only if'' part
  follows directly from Theorem~\ref{th:4}, which, in particular,
  implies that the inequality~\eqref{eq:2} holds for any equilibrium
  allocation $(\alpha^m_1)$.

  Conversely, assume that~\eqref{eq:2} holds for some allocation
  $(\alpha^m)$ of $\Lambda$. From Lemma~\ref{lem:2} we obtain the
  bound~\eqref{eq:18}. The dominated convergence theorem and the
  continuity of the measurable fields $\pi^m(\cdot,\Lambda)$ and
  $U_c(\cdot,\Lambda)$ on $\Sigma^M$ imply the continuity of the
  functions
  \begin{displaymath}
    f_m(w) \set \mathbb{E}[U_c(w,\Lambda)(\pi^m(w,\Lambda) -
    \alpha^m_0)],\quad w\in \Sigma^M.  
  \end{displaymath}
  In view of~\eqref{eq:7} and as $\alpha^m_0\not=0$ the functions
  $(f_m)$ satisfy~\eqref{eq:19} and since
  \begin{displaymath}
    \sum_{m=1}^M \alpha^m_0 = \sum_{m=1}^M \pi^m(w,\Lambda) = \Lambda, 
  \end{displaymath}
  they also satisfy~\eqref{eq:20}.

  Lemma~\ref{lem:3} then implies the existence of $\widehat w\in
  \interior{\Sigma^M}$ such that~\eqref{eq:21} holds.  The result now
  follows from Theorem~\ref{th:4} where the integrability
  condition~\eqref{eq:12} for $w=\widehat w$ holds because of
  Lemma~\ref{lem:2}.
\end{proof}

\begin{proof}[Proof of Theorem~\ref{th:2}.]
  The existence follows from Corollary~\ref{cor:1}.  To verify the
  uniqueness we define the \emph{excess utility} functions:
  \begin{align*}
    h_m(w) &\set \mathbb{E}[U'_m(\pi^m(w,\Lambda))(\pi^m(w,\Lambda) -
    \alpha^m_0)] \\
    &= \frac1{w^m}\mathbb{E}[U_c(w,\Lambda)(\pi^m(w,\Lambda) -
    \alpha^m_0)],\quad w\in \interior{\mathbf{R}_+^M}.
  \end{align*}
  By Lemma~\ref{lem:2}, these functions are well-defined and finite.
  The condition~\eqref{eq:3}, the properties~\ref{item:1}
  and~\ref{item:2} for $\pi^m=\pi^m(w,c)$, and the fact that the
  stochastic field $U'_m=U'_m(c)$ is strictly decreasing imply that
  \begin{enumerate}[label=(B\arabic{*}), ref=(B\arabic{*})]
    \setcounter{enumi}{\value{item}}
  \item \label{item:4} For every $m=1,\dots,M$, the function $h_m$ is
    $0$-homogeneous: $h_m(yw) = h_m(w)$, $w\in
    \interior{\mathbf{R}^M_+}$, $y>0$.
  \item \label{item:5} For $m\not=m_0$, where $m_0$ is the index for
    which~\eqref{eq:3} may not hold, the function $h_m$ has the
    \emph{gross-substitute} property: $h_m = h_m(w^1,\dots,w^M)$ is
    strictly decreasing with respect to $w^l$, $l\not=m$.
  \end{enumerate}

  Theorem~\ref{th:4} and the property~\ref{item:3} for
  $\pi^m=\pi^m(w,c)$ yield the uniqueness if we can show that the
  equation
  \begin{displaymath}
    h_m(w) = 0,\quad m=1,\dots,M,
  \end{displaymath}
  has no multiple solutions on $\interior{\Sigma^M}$.  Let $w_1$ and
  $w_2$ be two such solutions, $w_1\not=w_2$.  Assume that $w^{m_0}_1
  \leq w^{m_0}_2$; of course, this does not restrict any
  generality. Then
  \begin{displaymath}
    y \set \max_{m=1,\dots,M} \frac{w^m_1}{w^m_2} >1,  
  \end{displaymath}
  and there is an index $l_0\not=m_0$ such that $yw^{l_0}_2 =
  w^{l_0}_1$.  From \ref{item:4} we obtain that
  \begin{displaymath}
    h_m(yw_2) = h_m(w_2) = 0, \quad m=1,\dots,M,
  \end{displaymath}
  while from the gross-substitute property in~\ref{item:5} that
  \begin{displaymath}
    h_{l_0}(yw_2) < h_{l_0}(w_1) = 0. 
  \end{displaymath}
  This contradiction concludes the proof.
\end{proof}

\begin{Acknowledgments}
  The author thanks Kim Weston for comments and corrections. 
\end{Acknowledgments}

\bibliographystyle{plainnat}

\bibliography{../bib/finance}

\end{document}